\documentclass[11pt]{amsart}

\theoremstyle{plain}
\newtheorem{thm}{Theorem}[section]
\newtheorem{cor}[thm]{Corollary}

\newcommand{\R}{\mathbb{R}}
\newcommand{\C}{\mathbb{C}}

\numberwithin{equation}{section}

\begin{document}

\title{The modulus of Whittaker functions}

\author{Hans Volkmer}
\address{
Department of Mathematical Sciences,\\ University of
Wisconsin--Milwaukee, P.~O.~Box 413,\\
Milwaukee, WI 53201, U.S.A.}

\begin{abstract}
The paper discusses some properties of the modulus $|W_{k,m}(z)|$ of the Whittaker function $W_{k,m}(z)$.
In particular, completely monotone functions expressed in terms of $|W_{k,m}(z)|$ are  found. The results follow from
an integral representation for products of Whittaker functions due to Erd\'elyi (1938).
\end{abstract}

\keywords{Whittaker functions; Bessel functions; completely monotone functions}
\subjclass[2010]{33C15, 33C10}
\maketitle

\section{Introduction}
Recently, Boris Belinskiy asked the author the following question: Consider the quotient of Whittaker functions
\begin{equation}\label{1:q}
\frac{W_{k,m}(\tau z)}{W_{k,m}(z)} ,
\end{equation}
where $k$ is purely imaginary, $m$ is real, $|\arg z|<\frac12\pi$, $\tau>1$. Show that the quotient \eqref{1:q} is bounded for fixed $z$ and $\tau$ as a function of $k,m$. This problem occurs in Belinskiy's work in scattering theory.
A straightforward approach to this problem is to investigate the asymptotic behavior of the Whittaker function $W_{i r,m}$ as the real variables $r,m$
tend to infinity. However, this approach faces some difficulties. It appears that the required asymptotic results are not available.
The handbook of mathematical functions \cite[Chapter 13]{NIST} does give asymptotic results where both $k$ and $m$ tend to infinity but they do not include imaginary $k$. Olver \cite[page 401]{O} has such results but only for fixed $m$. Moreover, it will be difficult to obtain explicit bounds for \eqref{1:q} by using asymptotic methods.

We show in this paper that the quotient \eqref{1:q} can be estimated in a very simple way using an integral representation for products of Whittaker functions
due to Erd\'elyi \cite{E}.
We also show that the function
\[  t\mapsto t^{-1}e^{t\Re x}\left|W_{k,m}(tx)\right|^2,\quad t>0 \]
is completely monotone for appropriate values of $k,m$ and $x$.

Erd\'elyi's formula is known for a long time. However, the conclusions we draw from it appear to be new.

\section{Erd\'elyi's integral formula}
Erd\'elyi \cite{E} used the convolution theorem for the Laplace transform to derive the following integral representation for products of
Whittaker functions.

\begin{thm}\label{2:t1}
Suppose $x,y\in\C$, $|\arg x|<\pi$, $|\arg y|<\pi$, $t>0$, $m\in\C$, $k,l\in\C$ with $\Re(1-k-l)>0$.
Then we have
\begin{eqnarray}
&& (t^2xy)^{-1/2} e^{\frac12t (x+y)} W_{k,m}(tx)W_{l,m}(ty)\nonumber\\
&=& \frac{(xy)^m}{\Gamma(1-k-l)} \int_0^\infty e^{-tu} (x+u)^{k-\frac12-m}(y+u)^{l-\frac12-m} u^{-k-l}\label{2:e1}\\
&&\times {}_2F_1\left(\tfrac12+m-k,\tfrac12+m-l;,1-k-l;\frac{u(x+y+u)}{(x+u)(y+u)}\right)\,du.\nonumber
\end{eqnarray}
\end{thm}

Erd\'elyi \cite[(1) on page 873]{E} has \eqref{2:e1} with $t=1$. Buchholz \cite[(10) on page 89]{B} has \eqref{2:e1} including $t$.
It is easy to see that we can obtain the formula with $t$ from the special case $t=1$. For our purpose, we need \eqref{2:e1} including~$t$.
The function ${}_2F_1$ is the hypergeometric function. If $u=0$ we choose the principal value of ${}_2F_1$ and analytically continue this branch along the
path of integration. The powers appearing in \eqref{2:e1} are assigned their principal values.

If $m$ is real, $y=\bar x$, $l=\bar k$, we obtain the following result as a special case.
\begin{cor}\label{2:c}
Suppose $x\in\C$, $|\arg x|<\pi$,
$t>0$, $m\in\R$, $k\in\C$, $\Re k<\frac12$.
Then we have
\begin{eqnarray}
&& (t|x|)^{-1} e^{t\Re x} |W_{k,m}(tx)|^2\nonumber\\
&=& \frac{|x|^{2m}}{\Gamma(1-2\Re k)} \int_0^\infty e^{-tu} \left|(x+u)^{k-\frac12-m}\right|^2 u^{-2\Re k}\label{2:e2}\\
&&\times {}_2F_1\left(\tfrac12+m-k,\tfrac12+m-\bar k;,1-2\Re k;\frac{u(2\Re x+u)}{|x+u|^2}\right)\,du.\nonumber
\end{eqnarray}
\end{cor}
The argument $\frac{u(2\Re x+u)}{|x+u|^2}$ of the hypergeometric function in \eqref{2:e2} is real and less than $1$.

\section{Consequences}
We use Corollary \ref{2:c} to estimate \eqref{1:q}.

\begin{thm}\label{3:t1}
Suppose $x\in\C$, $|\arg x|\le \frac12\pi$, $m\in\R$, $k\in\C$, $\Re k<\frac12$ and $\tau\ge 1$.
Then
\begin{equation}\label{3:q}
 \left|\frac{W_{k,m}(\tau x)}{W_{k,m}(x)}\right|^2\le \tau e^{(1-\tau)\Re x} .
\end{equation}
\end{thm}
\begin{proof}
Since $\Re x\ge 0$, $\frac{u(2\Re x+u)}{|x+u|^2}$ lies between $0$ and $1$. Also, the hypergeometric power series appearing in \eqref{2:c} has
nonnegative coefficients. Therefore, the right-hand side of \eqref{2:c} is a decreasing function of $t$, and we obtain, for
$t=\tau\ge 1$,
\[ (\tau |x|)^{-1} e^{\tau\Re x} |W_{k,m}(\tau x)|^2\le |x|^{-1} e^{\Re x} |W_{k,m}(x)|^2 .\]
This implies \eqref{3:q}
\end{proof}

It is important that the bound on the right-hand side of \eqref{3:q} is independent of $k$ and $m$ in order to answer the question raised in the introduction.

Consider the hypergeometric function
\[  F(z)= {}_2F_1\left(\tfrac12+m-k,\tfrac12+m-\bar k;1-2\Re k;z\right) \]
under the assumptions on $k,m$ from Corollary \ref{2:c}.
Then $F(z)>0$ for $z\in[0,1)$. If $F(z)$ has a zero $z<0$ then let $p=p(k,m)$ denote the largest negative zero of $F(z)$. If there is no negative zero, set $p=-\infty$.

\begin{thm}\label{3:t2}
Suppose that $k\in\C$, $\Re k<\frac12$, $m\in\R$, and $x\in\C$ such that $|\arg x|<\pi$ if $p=-\infty$, and otherwise $|\arg x| \le \theta$, where $\theta$ is the angle between $\frac12\pi$ and $\pi$ for which
\[\tan\theta =-\frac{1}{(-p(k,m))^{1/2}}.\]
Then the function
\begin{equation}\label{3:f}
 f(t)=t^{-1} e^{t\Re x}|W_{k,m}(tx)|^2
\end{equation}
is (strictly) completely monotone on $(0,\infty)$, that is $(-1)^n f^{(n)}(t)>0$ for all $t>0$ and $n=0,1,2,\dots$
\end{thm}
\begin{proof}
If $|\arg x|\le \frac12\pi$ then \eqref{2:e2} shows that $f(t)$ is the Laplace transform of a positive function, so the statement follows.
If $\frac12\pi <|\arg x|<\pi$, then the function
\[ u\mapsto \frac{u(2\Re x+u)}{|x+u|^2},\quad u\ge 0, \]
attains an absolute minimum at $u=-\Re x$ with value $-\left(\frac{\Re x}{\Im x}\right)^2$. Under the assumption $|\arg x|\le \theta$, this minimum value is
greater than or equal to $p(k,m)$. Therefore, $f(t)$ is still the Laplace transform of a positive function (which has one zero when $|\arg x|=\theta$.)
\end{proof}

We should mention that the condition $|\arg x|\le \theta$ is sharp. If this condition is not satisfied (with $|\arg x|<\pi$) then $f(t)$ is not completely monotone because of Bernstein's theorem \cite[Thm 12b, page 161]{W}.
As a consequence of Theorem \ref{3:t2} we obtain that the function $W_{k,m}(x)$ has no zeros in the sector $|\arg x|\le \theta$.
It is clear that Theorem \ref{3:t1} could also be extended to the sector $|\arg x|\le \theta$.

We identify a case when $p(k,m)=-\infty$.

\begin{thm}\label{3:t3}
Suppose that $k<\frac12$,
\[ k-\frac12\le m \le\frac12-k \]
and $x\in\C$, $|\arg x|<\pi$. Then the function
\eqref{3:f} is (strictly) completely monotone on $(0,\infty)$.
\end{thm}
\begin{proof}
The integral representation \cite[15.6.1]{NIST} shows that ${}_2F_1(a,b;c;z)>0$ when $a,b,c\in\R$, $c>b>0$ and $z<1$.
This remains true if either $b=0$ or $b=c$. It follows that Theorem \ref{3:t2} applies with $p(k,m)=-\infty$.
\end{proof}

It is of interest to look at these results for the special case of the modified Bessel function $K_\nu$ connected to Whittaker functions
by
\[ K_\nu(x)=\left(\frac{\pi}{2x}\right)^{1/2} W_{0,\nu}(2x).\]

\begin{cor}\label{3:cor0}
Suppose $x\in\C$, $|\arg x|<\pi$, $t>0$, $\nu\in\R$. Then 
\begin{equation}
e^{2t\Re x}|K_\nu(tx)|^2=\pi(2|x|)^{2\nu} \int_0^\infty e^{-tu}|2x+u|^{-2\nu-1}{}_2F_1\left(\nu+\tfrac12,\nu+\tfrac12;1;\frac{u(4\Re x+u)}{|2x+u|^2}\right)du .
\end{equation}
\end{cor} 

\begin{cor}\label{3:c1}
Suppose $x\in\C$, $|\arg x|\le \frac12\pi$, $\nu \in\R$ and $\tau\ge 1$.
Then
\begin{equation}\label{3:q2}
 \left|\frac{K_\nu(\tau x)}{K_\nu(x)}\right|^2\le e^{(1-\tau)2\Re x} .
\end{equation}
\end{cor}

For $\nu\in\R$, let $z=p(\nu)$ denote the largest negative zero of the hypergeometric function
\[ {}_2F_1(\nu+\tfrac12,\nu+\tfrac12;1;z) \]
if it exists; otherwise set $p(\nu)=-\infty$.

\begin{cor}\label{3:c2}
Suppose that $\nu\in\R$, $x\in\C$ such that $|\arg x|<\pi$ if $p(\nu)=-\infty$, and otherwise $|\arg x|\le \theta$, where $\theta$ is the angle between $\frac12\pi$ and $\pi$ for which
\[\tan\theta =-\frac{1}{(-p(\nu))^{1/2}}.\]
Then the function
\begin{equation}\label{3:f2}
 f(t)=e^{2t\Re x}|K_\nu(tx)|^2
 \end{equation}
is (strictly) completely monotone on $(0,\infty)$.
\end{cor}

\begin{cor}\label{3:c3}
Suppose that $-\frac12\le \nu\le \frac12$, and $x\in\C$, $|\arg x|<\pi$. Then the function \eqref{3:f2}
is (strictly) completely monotone on $(0,\infty)$.
\end{cor}

Erd\'elyi  \cite[page 875]{E}  mentioned that the hypergeometric function appearing in connection with Bessel functions can be expressed as a
Legendre function by means of
\[  {}_2F_1(a,a;1;1-z)=z^{-\frac12 a}P_{-a}\left(\frac{2}{z}-1\right) .\]
If $\nu\in\R$ and $|\arg x|\le \frac12\pi$ then the function
\[ t\mapsto |K_\nu(tx)|^2 \]
is also completely monotone on $(0,\infty)$ because the product of completely monotone functions is again completely monotone.

\vspace{3mm}
\noindent
{\bf Example:}
Consider $\nu=2$. By numerical calculation we find $p(2)\approx -0.4573617040$. The
function
\[ t\mapsto e^{2t\Re x} |K_2(tx)|^2,\quad t>0\]
is completely monotone provided that $|\arg x|\le \theta$,
where
\[ \theta \approx 2.165428404 .\]

Boris Belinskiy also asked questions involving the Whittaker function $M_{k,m}$. For example, one would like to estimate
the quotient
\[
 \frac{M_{k,m}(z)}{M_{k,m}(\tau z)}
\]
when $k$ is purely imaginary, $m$ is real and $|\arg z|<\frac12 \pi$. However, there does not seem to exist a formula for $M_{k,m}$ analogous to
\eqref{2:e1}. Therefore, different methods would have to be used for the Whittaker function $M_{k,m}$.

\end{document}